\documentclass[hidelinks, onefignum, onetabnum]{siamart250211}

\usepackage{amsfonts}
\usepackage{amssymb}
\usepackage{mathtools}
\usepackage{subcaption}
\usepackage{booktabs}
\usepackage{tuda-pgfplots}
\pgfplotsset{
    tick label style={font=\small},
    label style={font=\small},
    legend style={font=\small}
}
\usetikzlibrary{external}
\usetikzlibrary{spy}
\usepgfplotslibrary{fillbetween}

\usepackage{notation}

\newsiamremark{remark}{Remark}
\newsiamremark{hypothesis}{Hypothesis}
\crefname{hypothesis}{Hypothesis}{Hypotheses}
\newsiamthm{claim}{Claim}
\newsiamremark{fact}{Fact}
\crefname{fact}{Fact}{Facts}

\headers{Learning parameterized stiff ODEs}{I. Cortes Garcia, P. Förster, W. Schilders and S. Schöps}

\title{Learning solutions of parameterized stiff ODEs using Gaussian processes\thanks{Submitted to the editors DATE.
\funding{This work is supported by the Graduate School CE within the Centre for Computational Engineering at Technische Universität Darmstadt and the ECSEL Joint Undertaking (JU) under grant agreement No. 101007319. The JU receives support from the European Union's Horizon 2020 research and innovation programme and the Netherlands, Hungary, France, Poland, Austria, Germany, Italy and Switzerland. Note that this work only reflects the authors' views and that the JU is not responsible for any use that may be made of the information it contains.}}}

\author{Idoia Cortes Garcia\thanks{Eindhoven University of Technology, Eindhoven, The Netherlands.}
\and Peter Förster\thanks{Technical University of Darmstadt, Darmstadt, Germany and Eindhoven University of Technology, Eindhoven, The Netherlands. (\email{peter.foerster@tu-darmstadt.de}).}
\and Wil Schilders\thanks{Eindhoven University of Technology, Eindhoven, The Netherlands.}
\and Sebastian Schöps\thanks{Technical University of Darmstadt, Darmstadt, Germany.}}

\ifpdf
\hypersetup{
  pdftitle={Learning solutions of parameterized stiff ODEs using Gaussian processes},
  pdfauthor={I. C. Garcia, P. Förster, W. Schilders and S. Schöps}
}
\fi

\begin{document}
\maketitle

\begin{abstract}
    Stiff ordinary differential equations (ODEs) play an important role in many scientific and engineering applications. Often, the dependence of the solution of the ODE on additional parameters is of interest, e.g.\ when dealing with uncertainty quantification or design optimization. Directly studying this dependence can quickly become too computationally expensive, such that cheaper surrogate models approximating the solution are of interest. One popular class of surrogate models are Gaussian processes (GPs). They perform well when approximating stationary functions, functions which have a similar level of variation along any given parameter direction, however solutions to stiff ODEs are often characterized by a mixture of regions of rapid and slow variation along the time axis and when dealing with such nonstationary functions, GP performance frequently degrades drastically. We therefore aim to reparameterize stiff ODE solutions based on the available data, to make them appear more stationary and hence recover good GP performance. This approach comes with minimal computational overhead and requires no internal changes to the GP implementation, as it can be seen as a separate preprocessing step. We illustrate the achieved benefits using multiple examples.
\end{abstract}



\section{Introduction}
\label{sec:int}
In recent years, machine learning (ML) has established itself as a valuable tool in scientific computing, in particular in the context of design optimization, uncertainty quantification or inverse problems. Within this setting, we want to focus on approximating a specific class of functions, namely solutions of stiff ordinary differential equations (ODEs) over time. It is well-known from time integration literature, see e.g.\ \cite{hairer1996}, that working with stiff ODEs requires extra care compared to their non-stiff counterparts and the same also holds true in the ML context.

Three major families of ML methods that have found attention in scientific computing are classical polynomial based methods, Gaussian processes (GPs) and neural networks (NNs), each with specific advantages and drawbacks. Our approach is based on GPs, however its key components can also be used in conjunction with NNs or polynomial based methods if desired.

In order to introduce our approach and differentiate it from existing ones, we consider a generic, parameterized initial value problem
\begin{align}
    \begin{split}
        \dd{t} \mb{x}(t, \mb{p}) &= \mb{f}(\mb{x}, t, \mb{p}), \quad t \in [t_0, t_\mr{f}]\\
        \mb{x}_0(\mb{p}) &= \mb{x}(t_0, \mb{p}),
    \end{split}
    \label{eq:ode}
\end{align}
where $\mb{x} \in \mathbb{R}^{n_x}$ and $\mb{p} \in \mathbb{R}^{n_p}$. Systems of stiff ODEs of the form of \cref{eq:ode}, often excluding the extra parameter dependence $\mb{p}$, have already been tackled from a variety of perspectives in scientific ML. Broadly speaking, one can classify the existing approaches into two groups. The first group only requires observations of the solution $\mb{x}(t, \mb{p})$ or its time derivative $\mb{f}(\mb{x}, t, \mb{p})$, whereas the second group additionally requires explicit knowledge about the underlying model, i.e.\ the functional form of $\mb{f}(\mb{x}, t, \mb{p})$. A prominent example for the latter type of approach are physics-informed NNs \cite{raissi2019}. In their standard form, they are known to struggle with stiff ODEs and multiple ideas have been developed to combat this issue \cite{ji2021, nasiri2022, de_florio2022}. Similar ideas in the context of GPs can be found in \cite{skilling1992, tronarp2019, chkrebtii2016}.

If the precise form of $\mb{f}(\mb{x}, t, \mb{p})$ is not known, one has to resort to different approaches. In the context of NNs, neural ODEs \cite{chen2018} learn the time derivative $\mb{f}(\mb{x}, t, \mb{p})$ based only on observations of $\mb{x}(t, \mb{p})$. They were tailored towards stiff problems in \cite{kim2021, fronk2024} and extended to parameterized systems in \cite{lee2021}. Related ideas for GPs are given in \cite{heinonen2018, ridderbusch2021}. In cases where evaluating $\mb{f}(\mb{x}, t, \mb{p})$ is not the bottleneck, or in settings where time integration is too expensive or one simply has no access to a time integration routine, learning the solution $\mb{x}(t, \mb{p})$ directly is the only option.

We aim to directly learn the solution $\mb{x}(t, \mb{p})$, however standard GP implementations often run into issues when trying to learn functions with regions of rapid variation, such as solutions of stiff systems. This is due to their inherent stationarity; roughly speaking they expect a function to have a similar level of variation along a given parameter direction, no matter which region of the parameter you look at. Various ideas have been proposed to cope with this issue, e.g.\ nonstationary \cite{paciorek2003, plagemann2008} and warped GPs \cite{snelson2003, snoek2014, marmin2018}. While these methods were explicitly designed to deal with nonstationary functions more efficiently, they require initial guesses of functional forms, the estimation of numerous hyperparameters or sampling algorithms such as Markov chain Monte Carlo (MCMC), which may hinder their general applicability. There are no real equivalents to these approaches for NNs, since deep NNs do not face the issue of inherent stationarity. Nonetheless, they also struggle to learn functions with regions of rapid variation, which motivates the domain decomposition inspired approach described in \cite{dolean2024}. A similar work in the context of GPs is given by \cite{gramacy2008}. (In the ML community these kinds of approaches are also known as mixture of experts models.) Although very promising, two potential issues remain with this last set of approaches. Firstly, if used in conjunction with a design of experiments workflow, which is typical for GPs, it can become costly to continually construct new decompositions on the fly. Secondly, anisotropic decompositions with many levels may be needed to effectively deal with regions of rapid variation that have complicated shapes, see \cref{fig:tdo} for an example that will be picked up again in \cref{sec:ex}.

\begin{figure}
    \begin{center}
        \includegraphics[width=0.47\textwidth]{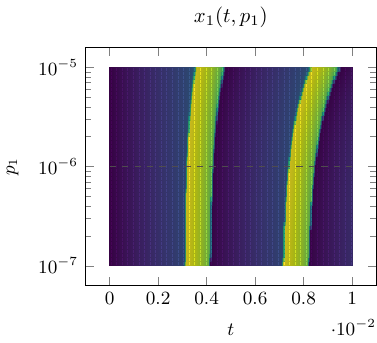} \hspace{\baselineskip} \includegraphics[width=0.47\textwidth]{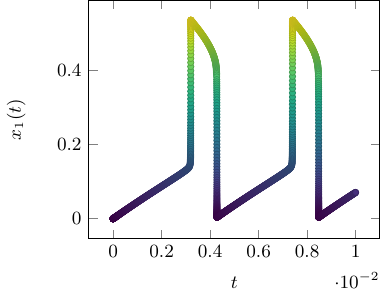}

        \vspace{0.3\baselineskip}
        \hspace{2.5\baselineskip} \includegraphics[width=0.47\textwidth]{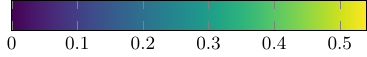}
    \end{center}
    \caption{Left: stiff ODE solution with regions of rapid variation that have complicated shapes. Right: one particular solution for $p_1 = 10^{-6}$ (indicated by the dashed line on the left) illustrating the extremely stiff behavior more closely.}
    \label{fig:tdo}
\end{figure}

We aim to solve this problem by combining ideas related to curve reparameterizations and dense output \cite{hairer1993} of time integration routines. More precisely, using a reparameterization $\tilde{\bs{\tau}}_\mb{p}^{-1}(s)$ we transform the solution curve $\mb{x}(t, \mb{p})$ into an approximately stationary one $\mbt{x}(s, \mb{p})$. This comes at the expense of having to learn two functions (the solution itself as well as the reparameterization), however we show that both of these functions are easier to learn by construction and we also prove that their smoothness, the key property for asymptotic convergence rate estimates, can be controlled by carefully choosing the reparameterization. The approach is similar in spirit to warped GPs, however it avoids the drawbacks of initial guesses of functional forms, extra hyperparameters or MCMC sampling by explicitly constructing the reparameterizations.

The following \cref{sec:gps} provides background knowledge on GPs, while \cref{sec:crs} introduces the reparameterizations and outlines our implementation. We then present examples showcasing the benefits of our approach in \cref{sec:ex}, before drawing conclusions in \cref{sec:cao}.

\section{Gaussian processes}
\label{sec:gps}
GPs are a probabilistic function approximation approach related to radial basis function interpolation \cite{benner2021}[Remark 9.8]. As such, one of their main advantages over other conventional function approximation approaches is their ability to seamlessly deal with unstructured data. Their probabilistic nature can also be beneficial for hyperparameter optimization or sequential experimental design, as will be discussed in more detail later. Similar to other kernel methods, GPs can converge rapidly depending on the smoothness of the approximated function and the given kernel \cite{fasshauer2012}. This makes them especially attractive in settings where acquiring data is expensive, in contrast to NNs, which often require comparably large amounts of data \cite{golestaneh2024}.

Before we introduce GPs more formally, let us fix the setting that we will be interested in later on. We will only consider learning a single solution component $x_i(t, \mb{p})$, $1 \leq i \leq n_x$ of the entire solution $\mb{x}(t, \mb{p})$ of \cref{eq:ode} at a time. While it is possible to use GPs to explicitly learn vector-valued functions \cite{bonilla2007}, such approaches often scale unfavorably in the number of components $n_x$, unless extra information about the relationships between the components is available. The case of linear relationships in particular is well understood, see e.g.\ \cite{swiler2020}, but for nonlinear relationships, as is the case for ODEs in general, there is no simple way to exploit extra knowledge, if it is present at all. Learning each component individually on the other hand is embarrassingly parallel, whereas most other parts of a GP workflow are sequential in nature. Furthermore, the solution data $\mb{x}(t, \mb{p})$, the computation of which is assumed to be the bottleneck, can be shared by all processes. We will therefore focus on this simplified setting in the following.

We now describe the key parts of a standard GP implementation; for a more complete overview see e.g.\ the textbook \cite{rasmussen2006}. For the purposes of illustration, let us first consider the one-dimensional case of learning a function (ODE solution) $x(t)$ based on observations $(t_i, x_i)$, $1 \leq i \leq N$. A GP for this setting is defined by (prior) mean $m: \mathbb{R} \to \mathbb{R}$ and covariance (kernel) functions $k: \mathbb{R} \times \mathbb{R} \to \mathbb{R}$. In practice, the mean function is often assumed to be zero as preprocessing measures such as standardization are common. A popular family of kernel functions is given by the Matérn kernels $k_\nu$, which vary in their smoothness from $\nu = 0$ all the way to $\nu = \infty$\footnote{In the literature, the parameter $\nu$ is often used slightly differently in the context of Matérn kernels, however we let it directly signify the smoothness of the kernel here, to keep a consistent and compact notation with the following sections.}. The two extreme cases coincide with the exponential ($k_0$) and Gaussian ($k_\infty$) kernels
\begin{align*}
    k_0(t, t') = \sigma^2 e^{-\frac{|t - t'|}{\theta}}, \qquad k_\infty(t, t') = \sigma^2 e^{-\frac{(t - t')^2}{2 \theta^2}},
\end{align*}
where $\sigma$ and $\theta$ are hyperparameters corresponding to the signal variance and length scale, respectively. Aside from the smoothness of the kernel, the length scale plays an important role in determining the approximation properties of the GP, which will become more apparent later.

\begin{figure}
    \begin{center}
        \begin{subfigure}{0.47\textwidth}
            \includegraphics[width=\textwidth]{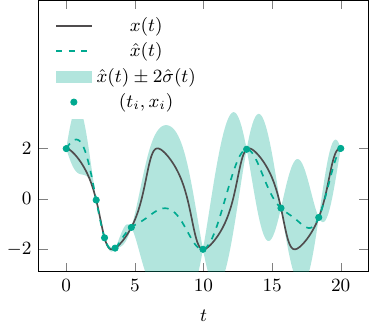}
            \caption{$p = 1$ and $N = 10$}
        \end{subfigure}
        \hspace{\baselineskip}
        \begin{subfigure}{0.47\textwidth}
            \includegraphics[width=\textwidth]{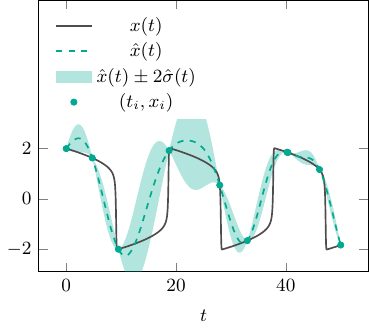}
            \caption{$p = 10$ and $N = 10$}
        \end{subfigure}

        \begin{subfigure}{0.47\textwidth}
            \includegraphics[width=\textwidth]{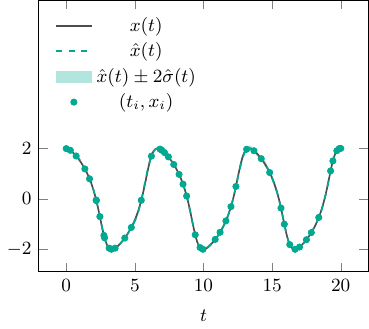}
            \caption{$p = 1$ and $N = 50$}
        \end{subfigure}
        \hspace{\baselineskip}
        \begin{subfigure}{0.47\textwidth}
            \includegraphics[width=\textwidth]{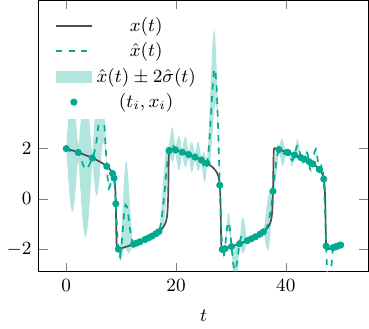}
            \caption{$p = 10$ and $N = 50$}
        \end{subfigure}
    \end{center}
    \caption{Exemplary sequential designs based on the $k_\infty$ Matérn (Gaussian) kernel for the VDPO with $p = 1$ (left) and $p = 10$ (right). The plots show the original solution $x(t)$, the prediction $\hat{x}(t)$, the posterior standard deviation $\hat{\sigma}(t)$ as well as the observations used for training $(t_i, x_i)$. The plots on the right illustrate the issues for stiff solutions.}
    \label{fig:conventional_doe}
\end{figure}

Being probabilistic models, GPs possess two key features that distinguish them from deterministic approaches such as radial basis function interpolation. First off, they do not only provide point predictions, but an entire posterior distribution, the mean $\hat{m}$ and covariance $\hat{k}$ of which are given by
\begin{align}
    \begin{split}
        \hat{m}(t) &= \mb{k}(t)^\T \mb{K}^{-1} \mb{x}\\
        \hat{k}(t, t') &= k(t, t') - \mb{k}(t)^\T \mb{K}^{-1} \mb{k}(t'),
    \end{split}
    \label{eq:posterior}
\end{align}
where we write $k_i(t) = k(t, t_i)$ for the entries of $\mb{k}(t)$ and $k_{ij} = k(t_i, t_j)$ for the entries of $\mb{K}$. The mean prediction $\hat{m}$ coincides with the radial basis function interpolant and in the following we will refer to it simply as the prediction and write $\hat{x}$ instead of $\hat{m}$. The second feature is the ability to analytically compute the log marginal likelihood $L$ and its derivatives w.r.t.\ the hyperparameters $\bs{\theta} = [\sigma, \theta]^\T$
\begin{align*}
    L(\bs{\theta}) &= -\frac{1}{2} \mb{x}^\T \mb{K}_{\bs{\theta}}^{-1} \mb{x} - \frac{1}{2} \log(\det \mb{K}_{\bs{\theta}}) - \frac{N}{2} \log(2 \pi)\\
    \pp{\theta_i} L(\bs{\theta}) &= \frac{1}{2} \big( \mb{K}_{\bs{\theta}}^{-1} \mb{x} \big)^\T \left( \pp{\theta_i} \mb{K}_{\bs{\theta}} \right) \mb{K}_{\bs{\theta}}^{-1} \mb{x} - \frac{1}{2} \tr \left( \mb{K}_{\bs{\theta}}^{-1} \pp{\theta_i} \mb{K}_{\bs{\theta}} \right),
\end{align*}
where we explicitly express the dependence of $\mb{K} = \mb{K}_{\bs{\theta}}$ on the hyperparameters. This allows for efficient gradient-based optimization of the hyperparameters.

Often, the probabilistic nature of GPs is further exploited by using the variance estimate in \cref{eq:posterior} or related design criteria \cite{marmin2018, harbrecht2021, burnaev2015} to formulate a sequential (adaptive) design of experiments, where only a few initial observations $(t_i, x_i)$ are provided to start off the design and additional observations are computed as requested by the design criterion. The aim of such a design is to avoid computing an unnecessarily large number of solutions to \cref{eq:ode}, as we assume a single solution to be very costly. An abstract workflow description for such an experimental design is as follows:
\begin{enumerate}
    \item select the initial observations,
    \item optimize the hyperparameters,
    \item evaluate the error estimate,
    \item optimize the design criterion to determine the next sample point,
    \item repeat steps 2-4 until the error estimate is below a desired tolerance.
\end{enumerate}
The left side of \cref{fig:conventional_doe} shows two intermediate steps of such a design for the solution of the following van der Pol oscillator (VDPO) \cite{hairer1993} with $p = 1$
\begin{align}
    \begin{split}
        \frac{\mr{d}^2}{\mr{d} t^2} x - p (1 - x^2) \dd{t} x + x = 0, \qquad x(0) = 2, \qquad \dd{t} x(0) = 0.
    \end{split}
    \label{eq:vdpo}
\end{align}

Extending the above workflow to the multidimensional case is straightforward via the use of product kernels
\begin{align*}
    k(\mb{p}, \mb{p}') = \sigma^2 \prod_{i=1}^n k_{\nu_i, \theta_i}(p_i, p_i'),
\end{align*}
where $n$ is the dimension of the input and the $k_{\nu_i, \theta_i}$ are one-dimensional kernels of smoothness $\nu_i$. Other extensions than this simple product construction are possible \cite{rasmussen2006}, however we will focus on these kinds of kernels in the following. The left of \cref{fig:conventional_errors} shows a comparison of the convergence of the relative $l^2$ error for designs based on kernels with varying smoothness $\nu$, for the VDPO with $p = 1$ as introduced in \cref{eq:vdpo}. While there is a visible difference between the designs based on low smoothness kernels ($\nu \leq 1$), the designs for $\nu \geq 2$ perform very similar.

\begin{figure}
    \begin{center}
        \includegraphics[width=0.47\textwidth]{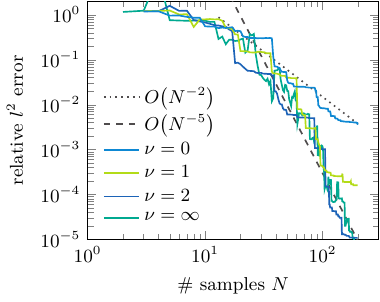} \hspace{\baselineskip} \includegraphics[width=0.47\textwidth]{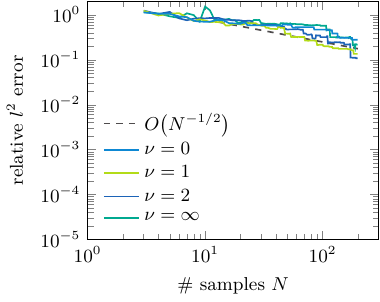}
    \end{center}
    \caption{Convergence of the sequential designs, both for the non-stiff case with $p = 1$ (left) and the stiff case with $p = 10$ (right).}
    \label{fig:conventional_errors}
\end{figure}

As an illustration for the kinds of problems that GPs face when trying to learn stiff solutions, we again consider the VDPO, this time however with $p = 10$. The increase in $p$ turns the problem into a stiff one, as is evident in the solution trajectories plotted on the right of \cref{fig:conventional_doe}. The bottom right plot also showcases the issue that the design runs into: although the predictions get better over large parts of the input domain, they fail to improve and even worsen in others, especially surrounding the regions of rapid variation. In essence this is due to the stationarity of the kernel, as a single length scale cannot simultaneously account for the regions of rapid and slow variation. The plot on the right hand side of \cref{fig:conventional_errors} shows that this holds independently of the kernel smoothness. Compared to the non-stiff case on the left, we lose somewhere between a factor of 4 to 10 in the rate of convergence.

\section{Curve reparameterizations}
\label{sec:crs}
As hinted at in the introduction, our approach to alleviating the issues described at the end of the previous \cref{sec:gps} is to reparameterize a solution component $x(t, \mb{p})$, such that we end up with a reparameterized solution $\tilde{x}(s, \mb{p})$ that is approximately stationary. More precisely, we want to bound the absolute value of the derivative w.r.t.\ the first variable, $\left\vert \dd{s} \tilde{x}(s, \mb{p}) \right\vert$, as this corresponds to the variation along the time direction. The key idea then is, that a tool for this already exists, namely reparameterizing the solution curve $[t,\ x(t, \mb{p})]^\T$ by arc length. To make this more explicit, we first recapitulate some basic facts about arc length reparameterizations of two-dimensonal curves.

Again dropping the parameter dependence at first, as in \cref{sec:gps}, we consider a function $x(t): [t_0, t_\mr{f}] \to \mathbb{R}$ that represents our stiff ODE solution, cf.\ \cref{eq:ode}. The arc length $\tau(t)$ of the curve $[t,\ x(t)]^\T$ is then defined as
\begin{align}
    \tau(t) = \int_{t_0}^t \sqrt{1 + \left( \dd{s} x(s) \right)^2}\, \mr{d}s = \int_{t_0}^t \sqrt{1 + f \big( x(s), s \big)^2}\, \mr{d}s, \quad t \in [t_0, t_\mr{f}].
    \label{eq:arclength}
\end{align}
Using the inverse function rule, this already gives
\begin{align*}
    \left\vert \dd{s} x \circ \tau^{-1}(s) \right\vert = \left\vert \frac{f \big( x \circ \tau^{-1}(s), \tau^{-1}(s) \big)}{\sqrt{1 + f \big( x \circ \tau^{-1}(s), \tau^{-1}(s) \big)^2}} \right\vert \leq 1,
\end{align*}
where the inverse $\tau^{-1}(s)$ exists for any $s \in [0, \tau(t_\mr{f})]$, which directly yields the desired result of bounding the derivative w.r.t.\ the first variable. In order to be able to always consider $s \in [0, 1]$ as the domain, we introduce a second map $\sigma: [0, 1] \to [0, \tau(t_\mr{f})]$. An important condition on $\sigma$ is that it be strictly monotonically increasing, so that the composition $\tau^{-1} \circ \sigma$ remains a valid reparameterization. For now, we simply choose a linear rescaling $\sigma(s) \coloneqq \tau(t_\mr{f})\, s$. This allows us to introduce
\begin{align*}
    \tilde{\tau}^{-1}(s) &\coloneqq \tau^{-1} \circ \sigma(s)\\
    \tilde{x}(s) &\coloneqq x \circ \tilde{\tau}^{-1}(s)\\
    \tilde{f}(s) &\coloneqq f \big( \tilde{x}(s), \tilde{\tau}^{-1}(s) \big),
\end{align*}
which yields
\begin{align}
    \left\vert \dd{s} \tilde{x}(s) \right\vert = \left\vert \sigma'(s) \frac{\tilde{f}(s)}{\sqrt{1 + \tilde{f}(s)^2}} \right\vert \leq \sigma'(s).
    \label{eq:xtp}
\end{align}

\Cref{fig:vdpo_reparameterization} illustrates the reparameterizations obtained using a linear $\sigma$ as described above. The zoom-in on the right shows a kink-like behavior that appears at the extreme points of the reparameterized solution in the stiff case, although this behavior is not present in the original solution (in light grey). This phenomenom can be explained by looking more closely at the left side of the inequality in \cref{eq:xtp}. For large derivatives $|\tilde{f}(s)| \gg 1$, we find
\begin{align*}
    \dd{s} \tilde{x}(s) \approx \sigma'(s) \frac{\tilde{f}(s)}{|\tilde{f}(s)|}.
\end{align*}
Now, it is important to notice that the derivative $|\tilde{f}|$ itself grows rapidly, such that the approximation also remains valid for points close to the roots of $\tilde{f}$. Thus while $\tilde{x}$ is smooth analytically, depending on the smoothness of $\tilde{f}$ of course, we may still observe numerical artifacts as highlighted in \cref{fig:vdpo_reparameterization}.

\begin{figure}
    \begin{center}
        \includegraphics[width=0.45\textwidth]{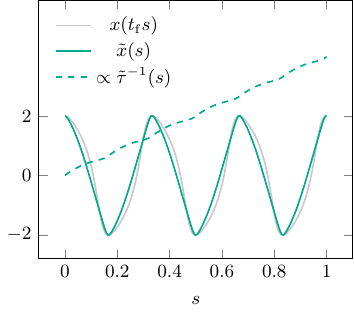} \hspace{\baselineskip} \includegraphics[width=0.45\textwidth]{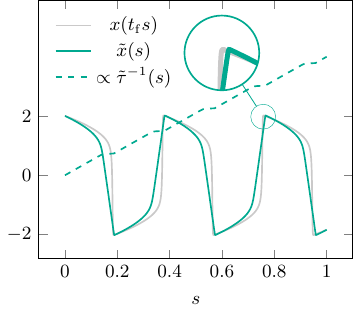}
    \end{center}
    \caption{Reparameterized solutions $\tilde{x}(s)$ using the linear definition of $\sigma$, both for the non-stiff case (left) and the stiff case (right). The original solutions are indicated in light grey with the reparameterized solutions in solid green and the reparameterizations $\tilde{\tau}^{-1}(s)$ as dashed green lines. The reparameterizations are scaled by a factor to fit the plots.}
    \label{fig:vdpo_reparameterization}
\end{figure}

\subsection{Controlling smoothness}
A simple way of addressing these numerical artifacts is given by the ideas of removable singularites from complex analysis or pole cancelation from control theory, respectively. Both amount to the same strategy of choosing $\sigma$ in such a way, that the roots of the Taylor expansion of $\sigma'$ cancel out the roots of the Taylor expansion of $\tilde{f}$. One straightforward approach is to define $\sigma$ using Hermite splines, such that $\sigma'$ has zeros of a given order at the roots of $\tilde{f}$. In practice the cancellation is of course not exact, due to machine precision and not being able to exactly compute the roots, but since $\tilde{x}$ is also not truly singular at these points, moderate precision suffices to achieve visibly and numerically acceptable results, see \cref{fig:vdpo_reparameterizations_splines} now and \cref{fig:vdpo_reparameterized_errors} in \cref{sec:ex} later.

\begin{figure}
    \begin{center}
        \begin{subfigure}{\textwidth}
            \includegraphics[width=0.45\textwidth]{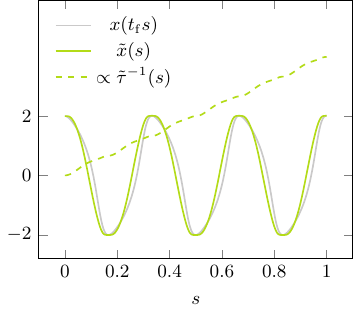} \hspace{\baselineskip} \includegraphics[width=0.45\textwidth]{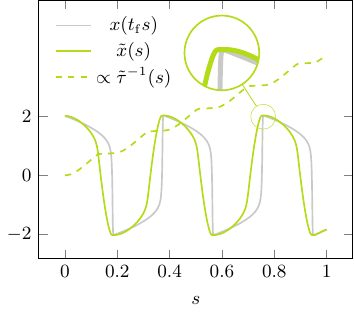}
            \caption{using $C^1$ cubic Hermite splines}
        \end{subfigure}

        \vspace{0.5\baselineskip}
        \begin{subfigure}{\textwidth}
            \includegraphics[width=0.45\textwidth]{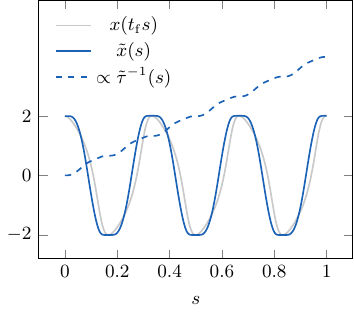} \hspace{\baselineskip} \includegraphics[width=0.45\textwidth]{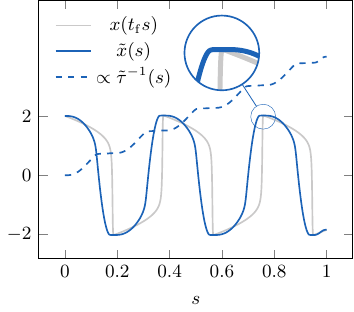}
            \caption{using $C^2$ quintic Hermite splines}
        \end{subfigure}
    \end{center}
    \caption{Reparameterized solutions using Hermite cubic and quintic splines for $\sigma$, arranged analogously to \cref{fig:vdpo_reparameterization}.}
    \label{fig:vdpo_reparameterizations_splines}
\end{figure}

When canceling only up to first order e.g., we can use a set of cubic Hermite splines \cite{curry1988} $H \in C^1$, uniquely defined by the following interpolation conditions
\begin{alignat*}{9}
    H(0) &= 0, &\quad H(s_{\mr{ext}_1}) &= \tau(t_{\mr{ext}_1}), &\quad H(s_{\mr{ext}_2}) &= \tau(t_{\mr{ext}_2}), &\quad &\dots &\quad H(1) &= \tau(t_\mr{f})\\
    H'(0) &= 0, &\quad H'(s_{\mr{ext}_1}) &= 0, &\quad H'(s_{\mr{ext}_2}) &= 0, &\quad &\dots &\quad H'(1) &= 0,
\end{alignat*}
where $t_{\mr{ext}_1}, t_{\mr{ext}_2}, \dots$ are the locations of the extreme points of $x(t)$ and $s_{\mr{ext}_i} = \frac{t_{\mr{ext}_i} - t_0}{t_\mr{f} - t_0}$. We only consider the extreme points of $x(t)$, as this already suffices to determine those of $\tilde{x}(s)$, since
\begin{align*}
    \sigma(s_{\mr{ext}_i}) = H(s_{\mr{ext}_i}) = \tau(t_{\mr{ext}_i})
\end{align*}
by construction and thus
\begin{align*}
    \tilde{\tau}^{-1}(s_{\mr{ext}_i}) = \tau^{-1} \circ \sigma(s_{\mr{ext}_i}) = t_{\mr{ext}_i}.
\end{align*}
The extreme points of $x$ and $\tilde{x}$ therefore match, compare also \cref{fig:vdpo_reparameterizations_splines}. As stated earlier, $\sigma$ is also constrained to be strictly monotonically increasing. \Cref{prop:cubic_spline} ensures that the above defined cubic Hermite splines indeed fulfill this condition. To achieve cancellation up to second order using the same strategy, we can make use of quintic Hermite splines $H \in C^2$ to also specify the second derivatives of $\sigma$ at the extreme points. The setup remains almost identical and a proof of strict monotonicity in this case is given in \cref{prop:quintic_spline}.

As indicated in \cref{sec:gps}, the smoothness of $\tilde{x}$ has a strong influence on the convergence rate when learning, and the different choices for $\sigma$ of course impact this as well. The following short proposition describes this relationship; the proof is provided in \cref{pr:smoothness_1D}.
\begin{proposition}
    \label{prop:smoothness_1D}
    Assuming $x \in C^\nu$ and $\sigma \in C^\mu$, we find $\tilde{x} \in C^{\min(\nu, \mu)}$ for the reparameterized solution and $\tilde{\tau}^{-1} \in C^{\min(\nu, \mu)}$ for the reparameterization.
\end{proposition}

\subsection{Extension to multiple dimensions}
With an understanding of the one-dimensional case, the most important remaining question is how to extend the approach to the multidimensional setting. We first extend the arclength definition from \cref{eq:arclength} to obtain
\begin{align*}
    \tau_\mb{p}(t) = \int_{t_0}^t \sqrt{1 + \left( \pp{s} x(s, \mb{p}) \right)^2}\, \mr{d}s = \int_{t_0}^t \sqrt{1 + f \big( x(s, \mb{p}), s, \mb{p} \big)^2}\, \mr{d}s, \quad t \in [t_0, t_\mr{f}],
\end{align*}
where we purposefully disregard derivatives w.r.t.\ the parameters $\mb{p}$ under the square root. The inverse $\tau_\mb{p}^{-1}(s)$ of this map w.r.t.\ $t$ then takes the role of $\tau^{-1}(s)$ from before and we can again compose it with analogous extensions $\sigma_\mb{p}$, of any of the previously defined $\sigma$, to find
\begin{align}
    \begin{split}
        \tilde{\tau}_\mb{p}^{-1}(s) \coloneqq \tau_\mb{p}^{-1} \circ \sigma_\mb{p}(s)\\
        \tilde{x}(s, \mb{p}) \coloneqq x \big( \tilde{\tau}_\mb{p}^{-1}(s), \mb{p} \big).
    \end{split}
    \label{eq:reparameterization}
\end{align}
This further illustrates the benefit of using the $\sigma_\mb{p}$ for obtaining a uniform domain for learning both $\tilde{x}$ and $\tilde{\tau}_\mb{p}^{-1}$. As an example, \cref{fig:vdpo_reparameterizations_2D} shows the original and reparameterized solutions, as well as the reparameterizations, for two different intervals of the parameter $p$ for the VDPO from \cref{eq:vdpo} using the linear definition of $\sigma$. The left side shows the non-stiff case, where we observe only moderate change, as is to be expected by comparing with \cref{fig:vdpo_reparameterization}. The right side however shows the stiff case, where we can clearly see the effect of the reparameterization on $\tilde{x}$ in the second plot.

\begin{figure}
    \begin{center}
        \begin{subfigure}{0.47\textwidth}
            \includegraphics[width=\textwidth]{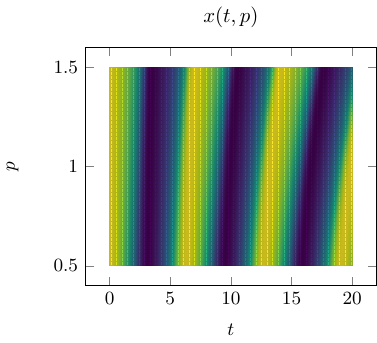}
        \end{subfigure}
        \hspace{0.3\baselineskip}
        \begin{subfigure}{0.47\textwidth}
            \includegraphics[width=\textwidth]{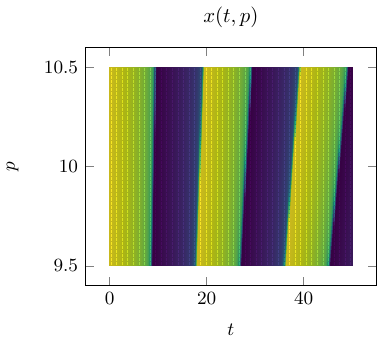}
        \end{subfigure}

        \begin{subfigure}{0.47\textwidth}
            \includegraphics[width=\textwidth]{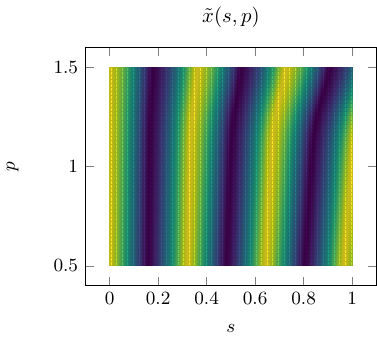}
        \end{subfigure}
        \hspace{0.3\baselineskip}
        \begin{subfigure}{0.47\textwidth}
            \includegraphics[width=\textwidth]{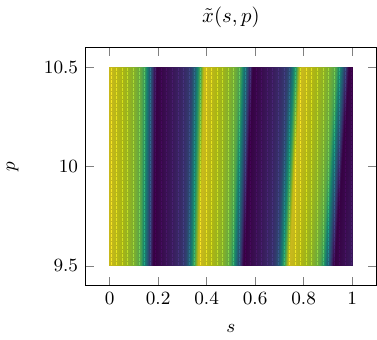}
        \end{subfigure}

        \begin{subfigure}{0.47\textwidth}
            \includegraphics[width=\textwidth]{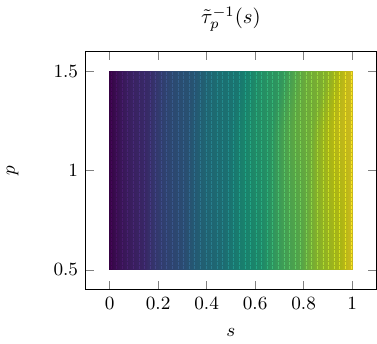}
        \end{subfigure}
        \hspace{0.3\baselineskip}
        \begin{subfigure}{0.47\textwidth}
            \includegraphics[width=\textwidth]{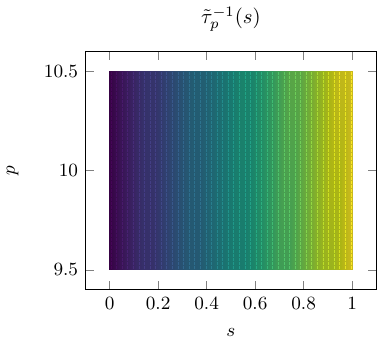}
        \end{subfigure}

        \vspace{0.3\baselineskip}
        \hspace{2.5\baselineskip} \includegraphics[width=0.47\textwidth]{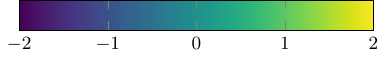}
    \end{center}
    \caption{From top to bottom: original solutions $x$, reparameterized solutions $\tilde{x}$ and reparameterizations $\tilde{\tau}_p^{-1}$ for both the non-stiff case (left) and the stiff case (right).}
    \label{fig:vdpo_reparameterizations_2D}
\end{figure}

The multidimensional case complicates the smoothness considerations slightly. The overall statements remain similar to \cref{prop:smoothness_1D} however, as captured in the following proposition. (We consider only one parameter $p$ for notational convenience and the proof is provided in \cref{pr:smoothness}.)
\begin{proposition}
    \label{prop:smoothness}
    Assuming $x(\cdot, p) \in C^\nu$, $f(x, t, \cdot) \in C^\kappa$ and $\sigma_{p}(\cdot) \in C^\mu$, we find $\tilde{x}(\cdot, p) \in C^{\min(\nu, \mu)}$ and $\tilde{x}(s, \cdot) \in C^\kappa$ for the reparameterized solution, as well as $\tilde{\tau}_p^{-1}(\cdot) \in C^{\min(\nu, \mu)}$ and $\tilde{\tau}_{\cdot}^{-1}(s) \in C^\kappa$ for the reparameterization, where $\cdot$ is used to denote variable arguments and explicitly stated arguments are considered fixed.
\end{proposition}

Before summarizing what the just defined concepts look like in terms of a numerical implementation, we give a brief remark to further distinguish our approach from existing methods based on warping.

\begin{remark}[Relation to warping]
    Warping of GP models can be grouped into two types. Classically, as described in\ \cite{snelson2003}, the warping acts on the \emph{solution}, i.e.\ the model learns $w \circ x(t, \mb{p})$ instead of $x(t, \mb{p})$ directly, where $w$ is an appropriately smooth, invertible function. In the second type, explored for instance in\ \cite{snoek2014, marmin2018}, the warping is applied to the \emph{input}, so that the model learns $x \circ w(t, \mb{p})$. In both cases, the warping function $w$ must be explicitly prescribed---typically as a linear combination of nonlinear basis functions such as $\tanh$---and involves additional hyperparameters that are optimized on top of the standard GP workflow.

    Our approach warps the input, but does not introduce an explicit parametric form of the warping function, nor any additional hyperparameters. Instead, it constructs the reparameterization directly from the available data points of $x(t, \mb{p})$. Consequently, no explicit choices about the form of $w$ are required, and the reparameterization is fully data-driven within the standard GP framework.
\end{remark}

\subsection{Numerical implementation}
In terms of implementation, one should of course try to match the smoothness of the reparameterization to that of the solution, if known. This is to achieve the maximum overall smoothness for learning while not wasting any effort, as the smoothness strongly influences the rate of convergence (compare \cref{fig:conventional_errors}). For solutions obtained via time integration e.g., dense output often provides an in some sense optimal interpolation that can guide this decision. We want to emphasize however, that the approach works with any data set that allows for an accurate one-dimensional interpolation along the direction containing the rapid variations. In any case, the following steps need to be executed to obtain the reparameterized solution and reparameterization for a fixed parameter combination $\mb{p}$:
\begin{enumerate}
    \item obtain $\tau_\mb{p}$ via time integration,
    \item find $\tau_\mb{p}^{-1}$ using monotonic splines to interpolate the pointwise inverse of $\tau_\mb{p}$,
    \item find the extreme points of the original solution $x$,
    \item define $\sigma_\mb{p}$ using the information from steps 1. and 3.
\end{enumerate}

\begin{remark}
    The first step boils down to a one-dimensional autonomous time integration problem, the cost of which is negligible to both the time integration of the full system, as well as the hyperparameter optimization and solution of linear systems during the learning process. Finding the extreme points of $x$ to moderate precision can be achieved by first finding the local extreme points in the discrete data and then executing a few iterations of a bracketing root finding scheme using the derivative of the accurate one-dimensional interpolant. The overall effort for obtaining the reparameterization is therefore negligible w.r.t.\ the learning process itself.
\end{remark}

\section{Examples}
\label{sec:ex}
Before we demonstrate the benefits of our approach on a number of examples, we briefly recall that we need to learn both the reparameterized solution $\tilde{x}$, as well as the reparameterization $\tilde{\tau}_\mb{p}^{-1}$, to be able to predict both the solution $\hat{\tilde{x}}$ and the physical time $\hat{t}(s) = \hat{\tilde{\tau}}_\mb{p}^{-1}(s)$. This also implies that predicting the solution at precise points in time requires an extra intermediate step. First, both predictions $\hat{\tilde{x}}(s, \mb{p})$ and $\hat{\tilde{\tau}}_\mb{p}^{-1}(s)$ need to be evaluated for a number of $s$ values, and then those predictions need to be interpolated to the precise point in time in a second step. In this case, it also becomes important for the prediction $\hat{\tilde{\tau}}_\mb{p}^{-1}(s)$ to be monotonic in $s$. This is of course true asymptotically, and already at relative $l^2$ errors of around $10^{-3}$ in practice, however it can also be ensured explicitly, e.g.\ by using one of the approaches described in \cite{swiler2020}. When predicting entire solution curves on the other hand, one can simply choose an appropriate equidistant sampling for $s$, since the reparameterization ensures that the predictions are then also sampled appropriately w.r.t.\ the dynamics.

The implementation is written in Julia \cite{bezanson2017}, openly available \cite{forster2025b}, and makes use of the DifferentialEquations.jl package for time integration \cite{rackauckas2017} and the DataInterpolations.jl package for the interpolation routines \cite{bhagavan2024}. In order to achieve a $C^2$ monotonic interpolant where necessary, quintic Hermite splines using finite difference estimates for the derivatives are used. In case the estimates violate the constraints presented in \cite{heß1994}, they are amended in a way to induce only small changes while still fulfilling the conditions.

\subsection{Van der Pol oscillator}
We begin by completing the results for the VDPO from \cref{eq:vdpo}. For the one-dimensional case, where we only considered two different values for the parameter $p$ to distinguish between non-stiff and stiff behavior, the reparameterizations for the different choices of $\sigma$ were already shown in \cref{fig:vdpo_reparameterization} and \cref{fig:vdpo_reparameterizations_splines}. We now also learn the different reparameterized solutions and reparameterizations, to then compare their convergence behavior with that of the conventional approach presented in \cref{fig:conventional_errors}. For this, we refer to the different options for $\sigma$ using their smoothness $\mu$, such that $\mu = \infty$ e.g.\ corresponds to the linear definition of $\sigma$. The kernel smoothness $\nu$ is selected in accordance with that of $\sigma$, except for the linear case, where we also consider $\nu = 0$, since the resulting reparameterized solution looks only $C^0$ visually.

\begin{figure}
    \begin{center}
        \begin{subfigure}{0.47\textwidth}
            \includegraphics[width=\textwidth]{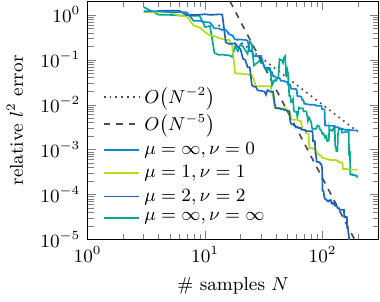}
            \caption{reparameterized solution $\hat{\tilde{x}}$ for $p = 1$}
            \label{fig:vdpo_reparameterized_errors_a}
        \end{subfigure}
        \hspace{\baselineskip}
        \begin{subfigure}{0.47\textwidth}
            \includegraphics[width=\textwidth]{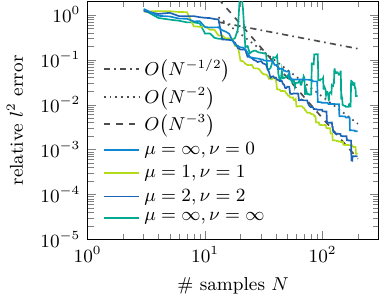}
            \caption{reparameterized solution $\hat{\tilde{x}}$ for $p = 10$}
            \label{fig:vdpo_reparameterized_errors_b}
        \end{subfigure}

        \begin{subfigure}{0.47\textwidth}
            \includegraphics[width=\textwidth]{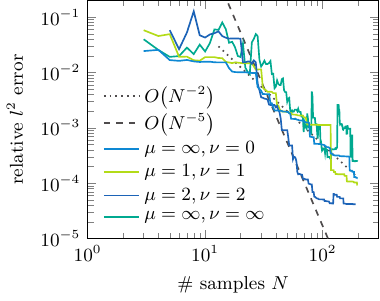}
            \caption{reparameterization $\hat{\tilde{\tau}}^{-1}$ for $p = 1$}
        \end{subfigure}
        \hspace{\baselineskip}
        \begin{subfigure}{0.47\textwidth}
            \includegraphics[width=\textwidth]{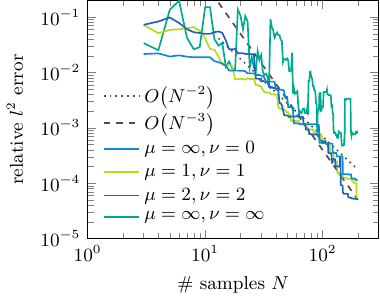}
            \caption{reparameterization $\hat{\tilde{\tau}}^{-1}$ for $p = 10$}
        \end{subfigure}
    \end{center}
    \caption{Convergence of sequential designs for the reparameterized solutions (top) and reparameterizations (bottom), based on different parameter values $p$ and different choices for $\sigma$ (indicated by $\mu$).}
    \label{fig:vdpo_reparameterized_errors}
\end{figure}

For the reparameterized solution $\tilde{x}$ in the non-stiff case with $p = 1$, \cref{fig:vdpo_reparameterized_errors_a} shows the expected behavior, as the results essentially match those for the conventional designs in \cref{fig:conventional_errors}. (The asymptotes are exactly the same.) This plot also showcases the importance of the different choices for $\sigma$, as the linear $\sigma$ does not suffice to match the conventional convergence behavior with a Gaussian ($\nu = \infty$) kernel. In the stiff case with $p = 10$, see \cref{fig:vdpo_reparameterized_errors_b}, the maximum observed rate decreases noticeably from 5 down to 3. However almost all rates, except for that with the Gaussian kernel, are still a factor of 4 to 6 better than those of the conventional designs in \cref{fig:conventional_errors}. To visually emphasize the magnitude of the improvement, \cref{fig:vdpo_reparameterized_errors_b} also contains an asymptote $O(N^{-1/2})$ that exactly matches the one from \cref{fig:conventional_errors}. Lastly, the bottom two plots show the corresponding results when learning the reparameterizations $\tilde{\tau}^{-1}$. The rates roughly match those for the reparameterized solutions $\tilde{x}$, albeit that the initial errors are smaller, note the changed $y$-axes compared to the top two plots.

\begin{figure}
    \begin{center}
        \begin{subfigure}{0.47\textwidth}
            \includegraphics[width=\textwidth]{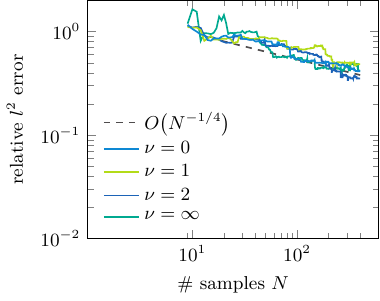}
            \caption{conventional solution $\hat{x}$}
        \end{subfigure}
        \hspace{\baselineskip}
        \begin{subfigure}{0.47\textwidth}
            \includegraphics[width=\textwidth]{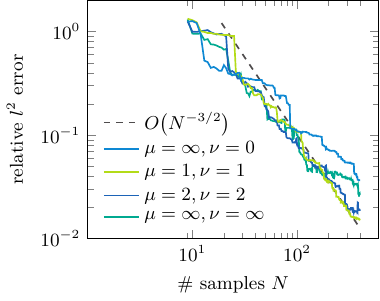}
            \caption{reparameterized solution $\hat{\tilde{x}}$}
        \end{subfigure}

        \begin{subfigure}{0.47\textwidth}
            \includegraphics[width=\textwidth]{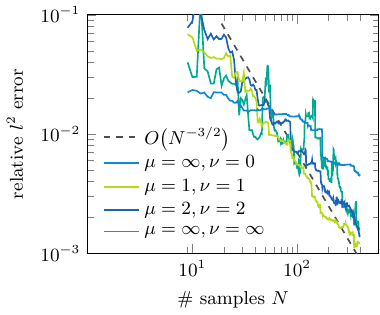}
            \caption{reparameterization $\hat{\tilde{\tau}}_p^{-1}$}
        \end{subfigure}
        \hspace{\baselineskip}
        \begin{subfigure}{0.47\textwidth}
            \includegraphics[width=\textwidth]{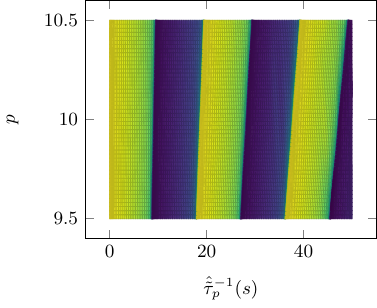}
            \caption{reparameterized prediction $\hat{\tilde{x}}(s, p)$}
        \end{subfigure}
    \end{center}
    \caption{Convergence results for conventional and reparameterized designs including the parameter dependence of the VDPO. Reparameterized prediction based on the design with smallest error ($\mu = 1$), compare the top right of \cref{fig:vdpo_reparameterizations_2D}.}
    \label{fig:vdpo_results_2D}
\end{figure}

Moving to the two-dimensional case, we now also consider the parameter dependence of the VDPO, in particular the interval leading to stiff behavior, for which we already presented a reference and the reparameterization with $\mu = \infty$ in \cref{fig:vdpo_reparameterizations_2D}. For this and all the following examples, we will always choose $\nu_i = \infty$, $2 \leq i \leq n_p$ in the product construction from \cref{sec:gps}. For the error estimate, we always select 10 random parameter combinations along with 100 random points in time each, to form the test data. The corresponding results are collected in \cref{fig:vdpo_results_2D}. Comparing the convergence rates, we observe a similar picture as for the stiff one-dimensional case, however all rates are approximately cut in half. The important thing to note is that all reparameterized designs, except for those based on the exponential kernel, attain rates approximately a factor of 6 greater than the conventional designs. In addition, the bottom right plot shows the reparameterized prediction based on the design with the smallest error ($\mu = 1$). A comparison with the top right plot in \cref{fig:vdpo_reparameterizations_2D} illustrates the visual agreement, in addition to the presented convergence plots.

\subsection{Tunnel diode oscillator}
\begin{figure}
    \begin{center}
        \begin{subfigure}{0.47\textwidth}
            \includegraphics[width=\textwidth]{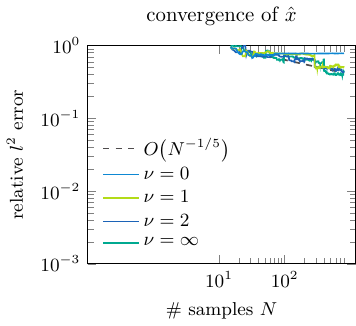}
            \caption{conventional solution $\hat{x}_1$}
        \end{subfigure}
        \hspace{\baselineskip}
        \begin{subfigure}{0.47\textwidth}
            \includegraphics[width=\textwidth]{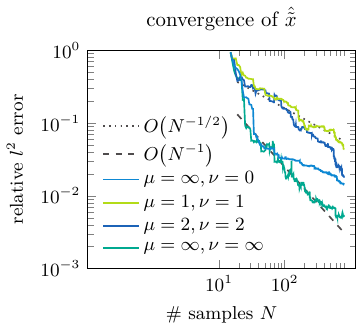}
            \caption{reparameterized solution $\hat{\tilde{x}}_1$}
        \end{subfigure}

        \begin{subfigure}{0.47\textwidth}
            \includegraphics[width=\textwidth]{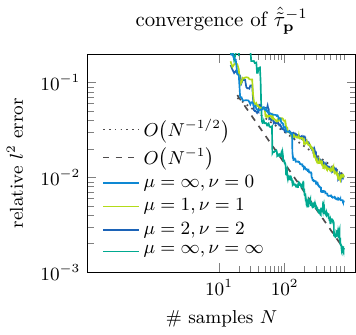}
            \caption{reparameterization $\hat{\tilde{\tau}}_\mb{p}^{-1}$}
        \end{subfigure}
        \hspace{\baselineskip}
        \begin{subfigure}{0.47\textwidth}
            \includegraphics[width=\textwidth]{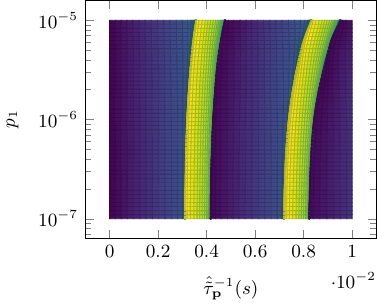}
            \caption{reparameterized prediction $\hat{\tilde{x}}_1(s, p_1)$}
        \end{subfigure}
    \end{center}
    \caption{Convergence results for conventional and reparameterized designs for the TDO. Reparameterized prediction based on the design with smallest error $(\nu = \infty)$, compare the left of \cref{fig:tdo}.}
    \label{fig:tdo_results_i_x=1}
\end{figure}

Stiff systems also arise in the context of circuit simulation. In particular, we look at the following ODE system with parameters $p_1$ and $p_2$ that models a tunnel diode oscillator (TDO)
\begin{align*}
    \dd{t} \begin{bmatrix}
        x_1\\
        x_2
    \end{bmatrix} = \begin{bmatrix}
        \frac{1}{p_1} \big( x_2 - g(x_1) x_1 \big)\\[2pt]
        \frac{1}{p_2} \left( \frac{1}{4} - x_1 - x_2 \right)
    \end{bmatrix}, \quad g(x_1) = 10.8\, x_1^2 - 8.766\, x_1 + 1.8.
\end{align*}
A reference for varying $p_1$ and fixed $p_2 = 3 \cdot 10^{-3}$ was already presented in the introduction in \cref{fig:tdo}. In the following, we learn the solution for $p_1$ varying as in the plot and $p_2 \in [2.9, 3.1] \cdot 10^{-3}$.

\Cref{fig:tdo_results_i_x=1} shows the results for the first solution component in a way similar to \cref{fig:vdpo_results_2D}. Both the convergence rates of the conventional designs in the top left, as well as those of the reparameterized designs end up slightly worse than for the VDPO, however the improvement still amounts to a factor of 5 for the best reparameterized designs. It is interesting to note that the reparameterized design with $\mu = \infty$ and $\nu = \infty$ clearly outperforms the others in this case. One possible explanation is that the original solution already appears only $C^0$ visually around the extreme points, as indicated on the right of \cref{fig:tdo}. The reparameterizations with $\mu < \infty$ might then do more harm than good, when trying to smooth out the behavior around said extreme points. Lastly, the bottom right plot shows the reparameterized prediction for the best design ($\nu = \infty$), to be compared with the left of \cref{fig:tdo}, which again emphasizes the visual agreement on top of the convergence plots.

A final point of interest may be the computational times, as we claimed that the reparameterization incurs essentially no overhead for the learning process. The present three-dimensional example now allows us to compare the average wall times of the different designs, as enough different reparameterizations needed to be computed, around 130 in total, for the different parameter combinations requested by the designs. In the end, the conventional designs took about $500\, \mr{s}$, whereas the reparameterized designs needed approximately $550\, \mr{s}$, supporting our argument that the extra effort for the reparameterizations is negligible, especially when compared to the provided benefits.

\subsection{Brusselator}
Lastly, we consider a classical example from chemical kinetics, the Brusselator \cite{hairer1993}. The system is given by
\begin{align*}
    \dd{t} \begin{bmatrix}
        x_1\\
        x_2
    \end{bmatrix} = \begin{bmatrix}
        p_1 - (p_2 + p_4) x_1 + p_3 x_1^2 x_2\\
        p_2 x_1 - p_3 x_1^2 x_2
    \end{bmatrix}
\end{align*}
with four parameters $p_1 \in [0.99, 1.01]$, $p_2 \in [2.97, 3.03]$, $p_3 \in [0.99, 1.01]$ and $p_4 \in [0.99, 1.01]$. \Cref{fig:brus_results_i_x=1} shows the results for the first solution component once more.

\begin{figure}
    \begin{center}
        \begin{subfigure}{0.47\textwidth}
            \includegraphics[width=\textwidth]{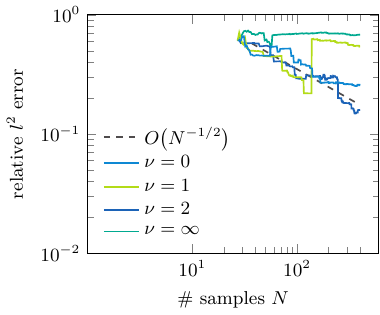}
            \caption{conventional solution $\hat{x}_1$}
        \end{subfigure}
        \hspace{\baselineskip}
        \begin{subfigure}{0.47\textwidth}
            \includegraphics[width=\textwidth]{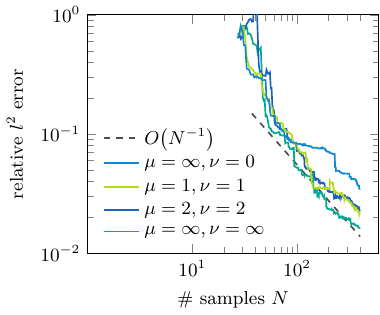}
            \caption{reparameterized solution $\hat{\tilde{x}}_1$}
        \end{subfigure}

        \begin{subfigure}{0.47\textwidth}
            \includegraphics[width=\textwidth]{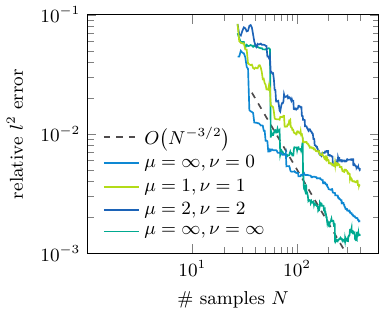}
            \caption{reparameterization $\hat{\tilde{\tau}}_\mb{p}^{-1}$}
        \end{subfigure}
        \hspace{\baselineskip}
        \begin{subfigure}{0.47\textwidth}
            \includegraphics[width=\textwidth]{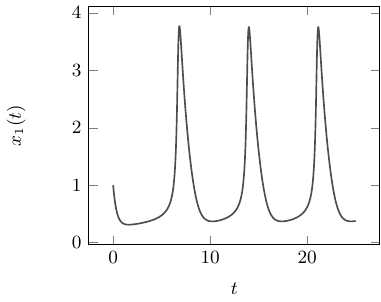}
            \caption{original solution $x_1$}
        \end{subfigure}
    \end{center}
    \caption{Convergence results for conventional and reparameterized designs for the Brusselator. Original solution for $\mb{p} = [1, 3, 1, 1]^\T$.}
    \label{fig:brus_results_i_x=1}
\end{figure}

We again observe significantly improved rates of convergence, this time however only by around a factor of 2-3, as the best rates of the conventional designs are noticeably better than for the previous examples ($1/2$ compared to $1/5$ and $1/4$ before). This can be explained by the observation that the Brusselator solution is, in a precise sense, less nonstationary than the previous examples. To quantify this, consider the following max-to-mean ratio (MMR) that measures the range of variation, and therefore the degree of nonstationarity, of a one-dimensional ODE solution
\begin{align*}
    r = |t_\mr{f} - t_0| \frac{\max_{t \in [t_0, t_\mr{f}]} \big\vert f \big( x(t), t \big) \big\vert}{\int_{t_0}^{t_\mr{f}} \big\vert f \big( x(s), s \big) \big\vert\, \mr{d}s}.
\end{align*}
A value of $r = 1$ indicates a completely stationary function and larger values of $r$ indicate more nonstationarity. \Cref{tab:nonstationarity} collects the MMRs of the original and reparameterized solutions for the different examples we considered and illustrates that the Brusselator is in fact significantly less nonstationary to begin with than the other examples, although it still clearly exhibits regions of rapid variation, as can be seen on the bottom right of \cref{fig:brus_results_i_x=1}. The table also suggests that the MMR can serve as a simple measure for testing whether the presented approach is worthwhile for a given problem.


\begin{table}
    \caption{Factors of improvement in the convergence rates for the different examples, depending on the MMRs of the original solutions ($r$) and those of the reparameterized solutions ($\tilde{r}$) with $\mu = \infty$.}
    \label{tab:nonstationarity}
    \begin{center}
        \begin{tabular}{lcccc}
            \toprule
            example & & $r$ & $\tilde{r}$ & factor\\
            \midrule
            VDPO ($p = 1$) & & 2.3 & 1.4 & 1\\
            Brusselator & & 11 & 2.7 & 2-3\\
            VDPO ($p = 10$) & & 35.5 & 5.9 & 6\\
            TDO & & 467.3 & 1 & 5\\
            \bottomrule
        \end{tabular}
    \end{center}
\end{table}

\section{Conclusions and outlook}
\label{sec:cao}
In this article, we introduced a workflow for reparameterizing stiff ODE solutions, so that they shall be learned efficiently with standard GP implementations. As argued in \cref{sec:crs}, the reparameterizations are easy and cheap to construct and they significantly improve the convergence rates (w.r.t.\ the number of samples) for a variety of examples, as illustrated in \cref{sec:ex}. The improvements vary depending on the problem, but mostly fall between a factor of 2-6 in the convergence rate if a kernel with smoothness larger than $\nu = 0$ is selected for the learning. \Cref{tab:nonstationarity} also showed that the differences between the examples can be attributed to different levels of nonstationarity in the ODE solutions.

Regarding future work, one interesting avenue is extending the approach to multiple reparameterized dimensions $\mb{t} = [t_1, t_2]^\T$, such that multidimensional regions of rapid variation may be handled as well. In principle, this could be achieved by simply stacking multiple one-dimensional reparameterizations, extending the definition from \cref{eq:reparameterization} to
\begin{align*}
    \tilde{\tau}_\mb{p}^{-1}(\mb{t}) = \begin{bmatrix}
        \tilde{\tau}_{t_2, \mb{p}}^{-1}(t_1)\\[3pt]
        \tilde{\tau}_{t_1, \mb{p}}^{-1}(t_2)
    \end{bmatrix}.
\end{align*}
Here, $\tilde{\tau}_{t_2, \mb{p}}^{-1}(t_1)$ denotes an inverse only w.r.t.\ $t_1$, analogous to the notation in \cref{sec:crs}, and treats $t_2$ as being fixed. As the learning process only adds single points $(\mb{t}_i, \mb{p}_i)$ at a time, we can thus compute the inverses one by one, such that the approach remains computationally feasible. While this still requires a fine grid of data for the reparameterized dimensions $\mb{t}$, the advantage of more easily considering additional parameters $\mb{p}$ remains. An extension of this form could e.g.\ be interesting in the context of partial differential equation solutions or parameter dependent bifurcations.

Separately, one could also improve upon learning each solution component individually by e.g.\ first performing a principal component analysis for identifying relationships in the solution data. Finally, it might be interesting to apply the approach to larger examples, to see how well it performs and to exploit further benefits of the GP setting, such as easy access to uncertainty quantification or (Bayesian) optimization.

\appendix
\section{Proofs}
\begin{proposition}
    \label{prop:cubic_spline}
    Given a cubic Hermite spline $H: [s_0, s_1] \to \mathbb{R}$, $s_0 < s_1$ with interpolation conditions of the form
    \begin{alignat*}{3}
        H(s_0) &= h_0, &\quad H(s_1) &= h_1\\
        H'(s_0) &= 0, &\quad H'(s_1) &= 0
    \end{alignat*}
    and $h_0 < h_1$, $H$ is strictly monotonically increasing.
\end{proposition}
\begin{proof}
    The derivative $H'$ is a quadratic polynomial with roots in $s_0$ and $s_1$ and since $h_0 < h_1$, there exists $s \in [s_0, s_1]$ with $H'(s) = \frac{h_1 - h_0}{s_1 - s_0} > 0$ by the mean value theorem, thus $H'(s) > 0$ for all $s \in [s_0, s_1]$.
\end{proof}

\begin{proposition}
    \label{prop:quintic_spline}
    Given a quintic Hermite spline $H: [s_0, s_1] \to \mathbb{R}$, $s_0 < s_1$ with interpolation conditions of the form
        \begin{alignat*}{3}
        H(s_0) &= h_0, &\quad H(s_1) &= h_1\\
        H'(s_0) &= 0, &\quad H'(s_1) &= 0\\
        H''(s_0) &= 0, &\quad H''(s_1) &= 0
    \end{alignat*}
    and $h_0 < h_1$, $H$ is strictly monotonically increasing.
\end{proposition}
\begin{proof}
    We prove the claim by computing the roots of $H$ and then using a similar argument as in the proof of \cref{prop:cubic_spline}. In the following, we work with $s_0 = 0$, $s_1 = 1$. The general case can be reverted back to this one by introducing an additional mapping $\varsigma(s) = \frac{s - s_0}{s_1 - s_0}$ and considering $H \circ \varsigma$ instead.

    We first state the ansatz and the first two derivatives for a quintic polynomial with coefficients $c_0, \dotsc, c_5 \in \mathbb{R}$
    \begin{align*}
        H(s) &= c_5 s^5 + c_4 s^4 + c_3 s^3 + c_2 s^2 + c_1 s + c_0\\
        H'(s) &= 5 c_5 s^4 + 4 c_4 s^3 + 3 c_3 s^2 + 2 c_2 s + c_1\\
        H''(s) &= 20 c_5 s^3 + 12 c_4 s^2 + 6 c_3 s + 2 c_2.
    \end{align*}
    Inserting the interpolation conditions, solving the resulting system for the coefficients and factorizing gives
    \begin{align*}
        H'(s) = 30 (h_1 - h_0) s^2 (s - 1)^2.
    \end{align*}
    Thus there are two double roots, in $0$ and $1$ or more generally in $s_0$ and $s_1$, and an analogous argument as in the proof of \cref{prop:cubic_spline} yields the claim.
\end{proof}

In the following, we provide the proof for \cref{prop:smoothness_1D}.
\begin{proof}
    \label{pr:smoothness_1D}
    For any $s \in [0, 1]$, the inverse function rule gives
    \begin{align*}
        \dd{s} \tilde{\tau}^{-1}(s) = \dd{s} \big( \tau^{-1} \circ \sigma(s) \big) = \frac{\sigma'(s)}{\tau'(t)\vert_{t = \tilde{\tau}^{-1}(s)}}
    \end{align*}
    and therefore $\tilde{\tau}^{-1} \in C^{\min(\nu, \mu)}$, since $\tau \in C^\nu$ by construction and $\tau' > 0$. The other claim then follows from $\tilde{x} = x \circ \tilde{\tau}^{-1}$.
\end{proof}

The following argument gives the proof for \cref{prop:smoothness}.
\begin{proof}
    \label{pr:smoothness}
    We first show that $\sigma_{\cdot}(s) \in C^\kappa$ for all extensions $\sigma_p(s)$ of the $\sigma$ described in \cref{sec:crs}. For this, we note that both for the linear definition, as well as for the definitions based on Hermite splines, $\sigma_p$ is always linear in $\tau_p$ and that there is no further dependence of $\sigma_p$ on $p$. The definition of $\tau_p$ then immediately yields this first claim.

    In a second step, we observe
    \begin{align*}
        \pp{s} \tilde{\tau}_p^{-1}(s) = \pp{s} \big( \tau_p^{-1} \circ \sigma_p(s) \big) = \frac{\pp{s} \sigma_p(s)}{\pp{t} \tau_p(t)\vert_{t = \tilde{\tau}_p^{-1}(s)}}
    \end{align*}
    based on the inverse function rule, which yields $\tilde{\tau}_p^{-1}(\cdot) \in C^{\min(\nu, \mu)}$, analogous to the previous proof. On the other hand we also have
    \begin{align*}
        \pp{p} \tilde{\tau}_p^{-1}(s) = \pp{p} \big( \tau_p^{-1} \circ \sigma_p(s) \big) = \frac{\pp{p} \sigma_p(s)}{\pp{p} \tau_p(t)\vert_{t = \tilde{\tau}_p^{-1}(s)}},
    \end{align*}
    which similarly gives $\tilde{\tau}_\cdot^{-1}(s) \in C^\kappa$. The remaining claims now follow from the definition of $\tilde{x}(s, p)$.
\end{proof}


\bibliographystyle{siamplain}
\bibliography{biblio}
\end{document}